\documentclass[a4paper,11pt]{amsart}
\usepackage{amsmath,amssymb,amsthm,leftidx,cleveref,xcolor}

\newcommand{\NN}{\mathbb{N}}
\newcommand{\ZZ}{\mathbb{Z}}
\newcommand{\RR}{\mathbb{R}}
\newcommand{\TT}{\mathbb{T}}
\newcommand{\FF}{\mathbb{F}}
\newcommand{\EE}{\mathbb{E}}

\newtheorem{theorem}{Theorem}[section]
\newtheorem{proposition}[theorem]{Proposition}

\newtheorem{definition}[theorem]{Definition}
\newtheorem{lemma}[theorem]{Lemma}

\newtheorem{remark}[theorem]{Remark}

\title{Cuplength estimates for time-periodic measures of Hamiltonian systems with diffusion}
\author{Oliver Fabert}
\thanks{O. Fabert, VU Amsterdam, The Netherlands. Email: o.fabert@vu.nl}
\begin{document}
\maketitle

\begin{abstract} We show how methods from Hamiltonian Floer theory can be used to establish lower bounds for the number of different time-periodic measures of time-periodic Hamiltonian systems with diffusion. After proving the existence of closed random periodic solutions and of the corresponding Floer curves for Hamiltonian systems with random walks with step width $1/n$ for every $n\in\NN$, we show that, after passing to a subsequence, they converge in probability distribution as $n\to\infty$. Besides using standard results from Hamiltonian Floer theory and about convergence of tame probability measures, we crucially use that sample paths of Brownian motion are almost surely H\"older continuous with H\"older exponent $0<\alpha<\frac{1}{2}$. \end{abstract}

\tableofcontents
\markboth{O. Fabert}{Time-periodic measures of diffusive Hamiltonian systems} 

\section{Random walks and the diffusion equation}\label{section 1}

Let $\Omega_n=(\Omega_n,\mathcal{P}(\Omega_n),\nu_n)$ with $\Omega_n=\FF_2^n$, $\FF_2=\{-1,+1\}$ denote the probability space for the $n$-fold coin-flip experiment equipped with the canonical counting measure $\nu_n$ as probability measure. Denoting by $\epsilon_i:\Omega_n\to\FF_2$ the projection onto the $i$.th component, we hence have for every $i=1,\ldots,n$ that $\nu_n(\{\omega_n\in\Omega_n:\epsilon_i(\omega_n)=\pm 1\})=\frac{1}{2}$ and $\epsilon_i$ is stochastically independent of $\epsilon_j$ for $j\neq i$. We define the discrete \emph{stochastic process} $W_n:\Omega_n\times [0,1]\to\RR$ as the rescaled linear interpolation of the $n$-fold \emph{random walk} with step width $\Delta t=\frac{1}{n}$ given by \[W_n(\omega_n,t)=\frac{1}{\sqrt{n}}\left(\sum_{i=1}^{\lfloor nt\rfloor} \epsilon_i(\omega_n)+\left(t-\frac{\lfloor nt\rfloor}{n}\right) \epsilon_{\lfloor nt\rfloor+1}(\omega_n)\right)\,\,\textrm{for every}\,\,\omega_n\in\Omega_n.\] While the expectation value of $W_n(\omega_n,1)$ is $\EE W_n(\omega_n,1)=0$, note that the rescaling factor $\frac{1}{\sqrt{n}}$ is chosen in order to guarantee that the variance is normalized to $\EE\left(W_n(\omega_n,1)-\EE W_n(\omega_n,1)\right)^2=1$.\\\indent 
It follows from the central limit theorem that $W_n(\cdot,1)$ converges \emph{in distribution} to $\mathcal{N}(0,1)$, the normal distribution with expectation value $0$ and variance $1$. In the same way one finds that for each $t\in [0,1]$ we have that $W_n(\cdot,t)$ converges in distribution to $\mathcal{N}(0,t)$. By the latter we mean that \[\nu_n(\{\omega_n\in\Omega_n:W_n(\omega_n,t)\leq a\})\,\,\textrm{converges to}\,\, \frac{1}{\sqrt{2\pi t}}\int_{-\infty}^a \exp\left(-\frac{x^2}{2t}\right)\,dx\] as $n\to\infty$. Note that the limiting distribution agrees with the fundamental solution of the diffusion (or heat) equation, so that the stochastic processes $W_n$ model diffusion as $n\to\infty$.\\\indent 
More precisely, it is known by the \emph{functional central limit theorem} that the stochastic processes $W_n:\Omega_n\to C^0([0,1],\RR)$ converge in distribution in the sense that the pushforward measures $\mu_n:=\nu_n\circ W_n^{-1}$ on $C^0([0,1],\RR)$ converge as Borel measures to a limit measure $\mu$, called the \emph{Wiener measure}. Denoting by $\rho_n$, $\rho$ the distribution corresponding to the Borel measure on $[0,1]\times\RR$ obtained as pushforward of the product measure $\lambda\otimes\mu_n$, $\lambda\otimes\mu$ ($\lambda$ = Lebesgue measure on $[0,1]$) under the canonical continuous map $[0,1]\times C^0([0,1],\RR)\to [0,1]\times\RR$, $(t,W)\mapsto (t,W(t))$, we find that $\rho_n$ converges in the distributional sense to $\rho$, the fundamental solution of the one-dimensional diffusion equation. For this observe that, for every $t\in [0,1]$, the Borel measure $\rho_n(t,\cdot)$ on $\RR$ is obtained as push-forward of the counting measure $\nu_n$ under the map $W_n(\cdot,t):\Omega_n\to\RR$,  by functoriality.\\\indent 
Generalizing from one-dimensional random walks $W_n:\Omega_n\to C^0([0,1],\RR)$ to $d$-dimensional random walks $W_n=(W_n^1,\ldots,W_n^d):\Omega_n^d\to C^0([0,1],\RR^d)$, $W_n(\omega_n^1,\ldots,\omega_n^d)=(W_n^1(\omega_n^1),\ldots,W_n^d(\omega_n^d))$ with one-dimensional random walks $W^1_n,\ldots,W^d_n$, we find that the limiting Borel measure $\rho$ on $[0,1]\times\RR^d$ now solves the diffusion equation on $[0,1]\times\RR^d$ given by  \[\frac{\partial\rho}{\partial t} = \frac{1}{2}\cdot\frac{\partial^2\rho}{\partial x^2}\,\,\textrm{with the Laplace operator}\,\,\frac{\partial^2}{\partial x^2}=\frac{\partial^2}{\partial x_1^2}+\ldots+\frac{\partial^2}{\partial x_d^2}.\]
\begin{remark}\label{Wiener_process}
The functional limit theorem can be rephrased by saying that the $n$-fold random walks $W_n:\Omega_n\times [0,1]\to\RR$ converge \emph{in distribution} as $n\to\infty$ to the \emph{Wiener process} $W:\Omega\times [0,1]\to\RR$, where  $\Omega=C^0([0,1],\RR)$ is equipped with the Wiener measure $\mu$ and one defines $W(\omega,t)=\omega(t)$ for all $(\omega,t)\in\Omega\times [0,1]$. For an intuitive understanding, note that, after replacing $\Omega$ by $\Omega_H=\FF_2^H$, $W$ agrees "up to an infinitesimal error" with $W_H$, where $H$ is an arbitrary "unlimited" (hyper)natural number number, i.e., the Wiener process is a random walk with "infinitesimal" step width $\frac{1}{H}$. Indeed this idea can be made fully rigorous in the framework of nonstandard analysis using the concept of Loeb measures, see \cite{An}.
\end{remark}

\section{Random Hamiltonian systems and the Fokker-Planck equation}
While the above relation between random walks and the diffusion equation can be generalized from $\RR^d$ to arbitrary Riemannian manifolds $Q$, based on the definition of the Laplace operator for Riemannian manifolds and using piecewise geodesic paths, in this paper we restrict our focus to random walks and diffusions on $\TT^d$ which are simply obtained by passing to the quotient in the target. Although the definition below can hence be generalized from $T^*\TT^d$ to the cotangent bundle $T^*Q$ of an arbitrary Riemannian manifold, let  
$H:\TT^1\times T^*\TT^d\to\RR$ be a time-periodic Hamiltonian function on the cotangent bundle of the $d$-dimensional torus $\TT^d=\RR^d/\ZZ^d$, where we set $H_t:=H(t,\cdot)$, and fix some diffusion constant $\sigma\in\RR$. 
\begin{definition}\label{stoch_sol}
Given $H$ and $\sigma$ as above, we call $u=(u_n)_{n\in\NN}$ with $u_n=(q_n,p_n):\Omega_n^d\times [0,1]\to T^*\TT^d$ and $u_n(\omega_n,1)=u_n(\omega_n,0)$ for every $\omega_n\in\Omega_n^d$ a \emph{sequence of closed random walk Hamiltonian orbits} if for every $n\in\NN$ and for every $\omega_n\in\Omega_n^d$ we have 
\begin{eqnarray*}
 q_n(\omega_n,t)-q_n(\omega_n,0)&=&\int_0^t \frac{\partial H_t}{\partial p}(q_n(\omega_n,\tau),p_n(\omega_n,\tau))\,d\tau + \sigma\cdot W_n(\omega_n,t),\\ \frac{\partial p_n}{\partial t}(\omega_n,t)&=&-\frac{\partial H_t}{\partial q}(q_n(\omega_n,t),p_n(\omega_n,t)).
 \end{eqnarray*}
\end{definition}
Analogous to \Cref{section 1}, every sequence of (closed) random walk Hamiltonian orbits $u=(u_n)_{n\in\NN}$ defines a sequence of Dirac measures $\mu_n^u:=\nu_n\circ u_n^{-1}$, $n\in\NN$ on $C^0([0,1],T^*\TT^d)$. Assume for the moment that $(u_n)_{n\in\NN}$ converges in distribution in the sense that the corresponding sequence of Dirac measures $(\mu_n^u)_{n\in\NN}$ converges as Borel measures to a Borel measure $\mu^u$ on $C^0([0,1],T^*\TT^d)$. Then the distribution $\rho^u$ on $[0,1]\times T^*\TT^d$, corresponding to the Borel measure obtained again as pushforward of $\lambda\otimes\mu^u$ under the canonical evaluation map $[0,1]\times C^0([0,1],T^*\TT^d)\to [0,1]\times T^*\TT^d$, satisfies the periodicity condition $\rho^u(0,\cdot)=\rho^u(1,\cdot)$ and is a solution of the following Hamiltonian version of the \emph{Fokker-Planck equation} (or \emph{forward Kolmogorov equation} or \emph{drift-diffusion equation})  \begin{equation}\label{Fokker-Planck}\frac{\partial\rho^u}{\partial t}=-\frac{\partial}{\partial q}\left(\frac{\partial H_t}{\partial p}\cdot \rho^u\right)+\frac{\partial}{\partial p}\left(\frac{\partial H_t}{\partial q}\cdot \rho^u\right)+\frac{\sigma^2}{2}\cdot\frac{\partial^2\rho^u}{\partial q^2}.\end{equation} Note that the latter equation combines the Hamiltonian version of the continuity equation modelling the change of $\rho^u$ under drift with the heat equation from \Cref{section 1} modelling the change of $\rho^u$ under diffusion. In order to see that $\rho^u$ is indeed a solution of the Hamiltonian Fokker-Planck equation, observe that 
the equation is equivalent to 
\[D_t\rho^u:=\frac{\partial\rho^u}{\partial t}+\frac{\partial H_t}{\partial p}\cdot \frac{\partial\rho^u}{\partial q}-\frac{\partial H_t}{\partial q}\cdot \frac{\partial\rho^u}{\partial p}=\frac{\sigma^2}{2}\cdot\frac{\partial^2\rho^u}{\partial q^2},\]
where $D_t\rho^u$ denotes the \emph{material derivative} describing the change of the $\rho^u$ under the influence of the drift given by the Hamiltonian vector field. 
\begin{remark} When $\sigma=0$ and when $u_n(\omega_n,\cdot)\equiv u=(q,p)$ for all $n\in\NN$, $\omega_n\in\Omega_n^d$, then $\mu_n^u=\mu^u=\delta_u$ is the Dirac measure localized at $u\in C^0([0,1],T^*\TT^d)$ and $\rho^u(t,q,p)=\delta(q-q(t))\cdot\delta(p-p(t))$, and it is immediate to check that $\rho^u$ solves \Cref{Fokker-Planck} with $\sigma=0$, that is, the continuity equation. \end{remark}    
\indent Since for every $n\in\NN$ the stochastic process $W_n$ extends from $[0,1]$ to the entire real line in such a way that $W_n(\cdot,t_2+1)-W_n(\cdot,t_1+1)$ and $W_n(\cdot,t_2)-W_n(\cdot,t_1)$ both agree in distribution for all $t_1\leq t_2$, $\rho^u$ indeed extends to a measure on $\RR\times T^*\TT^d$ satisfying the periodicity condition $\rho^u(t+1,\cdot)=\rho^u(t,\cdot)$ for all $t\in\RR$.
By generalizing methods from Hamiltonian Floer theory we show the following main result of this paper.
\begin{theorem}\label{main}
Assume that the time-periodic Hamiltonian $H:\TT^1\times T^*\TT^d\to\RR$ is of the form $H_t(q,p)=\frac{1}{2}|p|^2+F_t(q,p)$ with a smooth, time-periodic function $F_{t+1}=F_t$ with finite $C^1$-norm, and let $\sigma\in\RR$ be arbitrary. Then there exist $d+1$ = cuplength of (the loop space of) $\TT^d$ different sequences $u=(u_n)_{n\in\NN}$ of contractible closed random walk Hamiltonian orbits in the sense of \Cref{stoch_sol}. After passing to a suitable subsequence, they converge in distribution to $d+1$ different limiting time-periodic probability measures $\rho^u$ on $[0,1]\times T^*\TT^d$. In particular, we obtain at least $d+1$ different solutions of the corresponding Hamiltonian Fokker-Planck equation. 
\end{theorem}
Apart from the fact that we clearly expect that this statement can be proven for a larger class of Hamiltonians as long as they fulfill a suitable quadratic growth condition in the cotangent fibre, following the comment at the beginning of this section we also expect that the above theorem can be suitably generalized from time-periodic random walk Hamiltonian systems on $T^*\TT^d$ to those on the cotangent bundle $T^*Q$ of other Riemannian manifolds $Q$, possibly under additional restrictions such as the existence of a global orthonormal frame. In view of these broader questions about the interplay between Hamiltonian systems with random walks and solutions of Fokker-Planck equations on more general symplectic manifolds, this paper focuses on a proof using methods from Hamiltonian Floer theory, where our main aim is to illustrate how the weak notion of convergence in distribution and the limiting Brownian motion with its almost surely non-differentiable sample paths can still be incorporated in the analytical framework of Hamiltonian Floer theory. \emph{We would like to emphasize that this paper is written for researchers with a background in Hamiltonian Floer theory, in particular, no prior knowledge about stochastic processes is required.}
\begin{remark}
Following up on \Cref{Wiener_process}, our result can be used to establish the existence of $d+1$ different solutions of the Hamiltonian \emph{stochastic differential equation} 
\begin{eqnarray*}
 dq(\omega,t)&=&\frac{\partial H_t}{\partial p}(q(\omega,t),p(\omega,t))\,dt + \sigma\cdot dW_t(\omega),\\ \frac{\partial p}{\partial t}(\omega,t)&=&-\frac{\partial H_t}{\partial q}(q(\omega,t),p(\omega,t)).
 \end{eqnarray*}
Since a precise formulation of the statement as well as of the proof would require some substantial extra theoretical background, we decided to focus on the convergence of probability distributions. Using the convergence result in \Cref{main} combined again with nonstandard analysis methods, the aforementioned solutions however can again be obtained by replacing the natural numbers $n$ in the sequences of closed random walk Hamiltonian orbits by a suitable "unlimited" (hyper)natural number $H$ as in \Cref{Wiener_process}.
\end{remark}
\indent Let $n\in\NN$ be arbitrary. Instead of looking for closed random walk Hamiltonian orbits $u_n=(q_n,p_n):\Omega_n^d\times [0,1]\to T^*\TT^d$ in the sense of \Cref{stoch_sol}, we make use of the fact that we can equally well look for random Hamiltonian orbits  $\bar{u}_n=(\bar{q}_n,p_n):\Omega_n^d\times [0,1]\to T^*\TT^d$ with boundary condition $(\bar{q}_n,p_n)(\omega_n,1)=\phi^{\omega_n}_1( (\bar{q}_n,p_n)(\omega_n,0))$ for the $\omega_n$-dependent symplectic flow \[\phi^{\omega_n}_t: T^*\TT^d\to T^*\TT^d,\,\,\phi^{\omega_n}_t(q,p)=(q-\sigma\cdot W_n(\omega_n,t),p),\] solving the Hamiltonian equation 
\begin{equation}\label{G-eq} \frac{\partial\bar{q}_n}{\partial t}(\omega_n,t)=\frac{\partial K^{\omega_n}_t}{\partial p}(\bar{u}_n(\omega_n,t)),\,\,\frac{\partial p_n}{\partial t}(\omega_n,t)=-\frac{\partial K^{\omega_n}_t}{\partial q}(\bar{u}_n(\omega_n,t))\end{equation} for the $\omega_n$-dependent time-dependent Hamiltonian \[K^{\omega_n}_t=H_t\circ(\phi^{\omega_n}_t)^{-1}:T^*\TT^d\to\RR,\,\, K^{\omega_n}_t(\bar{q},p)=H_t(\bar{q}+\sigma\cdot W_n(\omega_n,t),p)\] with $K^{\omega_n}_{t+1}=K^{\omega_n}_t\circ(\phi^{\omega_n}_1)^{-1}$. Here the relation between $q_n$ and $\bar{q}_n$ is given by \[\bar{q}_n(\omega_n,t)=q_n(\omega_n,t)-\sigma\cdot W_n(\omega_n,t).\] Note that, as in classical Hamilton theory, the random Hamiltonian orbits $\bar{u}_n:\Omega_n^d\times [0,1]\to T^*\TT^d$ are precisely the critical points of the random symplectic action 
\[\EE\int_0^1 \left(p_n(\omega_n,t)\partial_t\bar{q}_n(\omega_n,t)-K^{\omega_n}_t(\bar{q}_n(\omega_n,t),p_n(\omega_n,t))\right)\,dt \] on the space of paths $(\bar{q}_n,p_n):\Omega_n^d\times [0,1]\to T^*\TT^d$ satisfying the $\omega_n$-dependent boundary condition $(\bar{q}_n,p_n)(\omega_n,1)=\phi^{\omega_n}_1((\bar{q}_n,p_n)(\omega_n,0))$, 
where $\EE$ denotes the expectation value with respect to the counting measure $\nu_n$ on $\Omega_n^d=\FF_2^{n\cdot d}$. As in the non-random setting, the $d+1$ different random Hamiltonian orbits as claimed in \Cref{main} are distinguished by their random symplectic action, by studying $L^2$-gradient flow lines of this symplectic action, also called Floer curves, see below.\\
\indent While $W_n(\omega_n,\cdot)$ is only continuous, each of the Hamiltonian orbits $\bar{u}_n(\omega_n,\cdot)$ can be assumed to be differentiable for each $\omega_n\in\Omega_n^d$. Hence every sequence of Hamiltonian orbits $\bar{u}=(\bar{u}_n)_{n\in\NN}$ defines a sequence of Dirac measures $(\bar{\mu}_n^{\bar{u}})_{n\in\NN}$ by setting $\bar{\mu}_n^{\bar{u}}:=\nu_n\circ \bar{u}_n^{-1}$ on $C^1([0,1],T^*\TT^d)$ for every $n\in\NN$. Apart from the existence result for sequences of random Hamiltonian orbits $\bar{u}=(\bar{u}_n)_{n\in\NN}$, the other main finding is that there is a subsequence that converges in distribution, that is, after passing to a suitable subsequence, the sequence of Borel measures $(\bar{\mu}_n^{\bar{u}})_{n\in\NN}$ converges to a limiting Borel measure $\bar{\mu}^{\bar{u}}$ on $C^1([0,1],T^*\TT^d)$. Using the continuous map $C^1([0,1],T^*\TT^d)\times C^0([0,1],\TT^d)\to C^0([0,1],T^*\TT^d)$, $(\bar{u},W)\mapsto \bar{u}+(\sigma\cdot W,0)$, the Borel measure $\mu^u$, obtained via pushforward of the Borel measure $\bar{\mu}^{\bar{u}}\otimes\mu$ ($\mu$ = Wiener measure), is the limit of the Dirac measures $\mu_n^u=\nu_n\circ u_n^{-1}$ given by $u=(u_n)_{n\in\NN}$. Since the random symplectic action 
\begin{eqnarray*} &&\EE\int_0^1 \left(p_n(\omega_n,t)\partial_t\bar{q}_n(\omega_n,t)-K^{\omega_n}_t(\bar{q}_n(\omega_n,t),p_n(\omega_n,t))\right)\,dt =\\&&\EE\int_0^1 \left(p_n(\omega_n,t)\partial_t\bar{q}_n(\omega_n,t)-H_t(q_n(\omega_n,t),p_n(\omega_n,t))\right)\,dt\end{eqnarray*} 
can be written as 
\begin{eqnarray*}
&&\int_{C^1}\int_0^1 p(t)\partial_t\bar{q}(t)\,dt\,d\bar{\mu}_n^{\bar{u}} - \int_{C^0} \int_0^1 H_t(q(t),p(t))\,dt\,d\mu_n^u =\\
&&\int_{C^0}\int_0^1 \left(p(t)\frac{\partial H_t}{\partial p}(q(t),p(t))- H_t(q(t),p(t))\right)\,dt\,d\mu_n^u =\\
&&\int_0^1\int_{T^*\TT^d} \left(p\frac{\partial H_t}{\partial p}(q,p)- H_t(q,p)\right)\,\rho^u_n(t,q,p) \,dp\,dq \,dt,
\end{eqnarray*} using the integral over all paths $(\bar{q},p)$ in $C^1([0,1],T^*\TT^d)$ equipped with the Dirac measure $\bar{\mu}^{\bar{u}}_n$ and over all paths $(q,p)$ in $C^0([0,1],T^*\TT^d)$ equipped with the Dirac measure $\mu^u_n$, respectively, after passing to the subsequence as above, the random symplectic actions converge to \
\begin{eqnarray*}
&&\int_{C^1}\int_0^1p(t)\partial_t\bar{q}(t)\,dt\,d\bar{\mu}^{\bar{u}}-\int_{C^0}\int_0^1 H_t(q(t),p(t))\,dt\,d\mu^u =\\
&&\int_{C^0}\int_0^1 \left(p(t)\frac{\partial H_t}{\partial p}(q(t),p(t))- H_t(q(t),p(t))\right)\,dt\,d\mu^u =\\
&&\int_0^1\int_{T^*\TT^d} \left(p\frac{\partial H_t}{\partial p}(q,p)- H_t(q,p)\right)\,\rho^u(t,q,p) \,dp\,dq \,dt,
\end{eqnarray*}
where the Dirac measures $\bar{\mu}^{\bar{u}}_n$, $\mu^u_n$, $\rho^u_n$ are replaced by the limiting Borel measure $\bar{\mu}^{\bar{u}}$, $\mu^u$, $\rho^u$ respectively. In order to show that the limiting Borel measures obtained from the $d+1$ different sequences of time-periodic random walk Hamiltonian orbits still can be distinguished using their symplectic actions, we show that the Floer curves used to distinguish the $d+1$ sequences of random walk Hamiltonian orbits converge as well, possibly after passing to a further subsequence, in a distributional Gromov-Floer sense.\\
\indent Apart from the use of fractional Sobolev spaces, the main technical input that we use is the following well-known result about the regularity of sample paths of Brownian motion, see (\cite{MY}, corollary 1.20).
\begin{proposition}\label{regularity}
With respect to the Wiener measure the following holds true: For every $0<\alpha<\frac{1}{2}$, a path $\omega\in\Omega^d=C^0(\RR,\RR^d)$ is almost surely H\"older continuous with H\"older exponent $\alpha$, i.e., there exists $c_\alpha>0$ such that $|\omega(t_2)-\omega(t_1)|\leq c_\alpha\cdot |t_2-t_1|^\alpha$ for every $t_1,t_2\in\RR$. In other words, the corresponding subspace of H\"older continuous functions has full Wiener measure.
\end{proposition}

\section{Hamiltonian Floer theory with diffusion}
The proof consists of the following steps, where for each $\omega_n\in\Omega_n^d$, $n\in\NN$ we set $q_n^{\omega_n}(t):=q_n(\omega_n,t)$ and $p_n^{\omega_n}(t):=p_n(\omega_n,t)$, where we refer to \cite{AD}, \cite{AH}, \cite{Sch} for more details on the underlying Hamiltonian Floer theory and \cite{DS} for the necessary modifications in the case of boundary conditions twisted by a symplectomorphism. \\

\noindent\emph{The case $F_t\equiv 0$:} In this case the diffusive Hamiltonian equations simplify to \[q_n^{\omega_n}(t)-q_n^{\omega_n}(0)=\int_0^t p_n^{\omega_n}(\tau)\,d\tau +\sigma\cdot W_n(\omega_n,t),\,\,\frac{\partial p_n^{\omega_n}}{\partial t}(t)=0,\] which is equivalent to \[p_n^{\omega_n}(t)=p_n^{\omega_n}(0),\,\, q_n^{\omega_n}(t)=q_n^{\omega_n}(0)+t\cdot p_n^{\omega_n}(0)+\sigma\cdot W_n(\omega_n,t),\,\,t\in [0,1].\] Furthermore $(q_n^{\omega_n},p_n^{\omega_n}):[0,1]\to T^*\TT^d$ satisfies the boundary condition $(q_n^{\omega_n},p_n^{\omega_n})(1)=(q_n^{\omega_n},p_n^{\omega_n})(0)$ precisely when \[p_n^{\omega_n}(0)-\sigma\cdot W_n(\omega_n,1)\in\ZZ^d,\,\,q_n^{\omega_n}(0)\in\TT^d\,\,\textrm{arbitrary.}\] Also taking into account that we are interested in contractible solutions, we arrive at \[p_n^{\omega_n}(0)=\sigma\cdot W_n(\omega_n,1),\,\,q_n^{\omega_n}(0)\in\TT^d\,\,\textrm{arbitrary.}\]
In particular, we note that there is an entire manifold of stochastic time-dependent solutions satisfying the boundary condition, and the cuplength estimates holds in the Morse-Bott sense, i.e., after adding a sufficiently small Morse function on $\TT^d$.\\

\noindent\emph{Moduli spaces of Floer curves:} 
In order to prove the existence of $d+1$ (= cuplength of $\TT^d$) different solutions $(\bar{q}_n^{\omega_n},p_n^{\omega_n}):[0,1]\to T^*\TT^d$ using Hamiltonian Floer theory, we now follow the standard strategy, see e.g. \cite{AH}, \cite{Sch}; since the details are standard as well, we only outline the key steps.\\
\indent Since the Hamiltonian $K_t^{\omega_n}=K^0+G_t^{\omega_n}$ with $K^0(q,p)=\frac{1}{2}|p|^2$, $G_t^{\omega_n}=F_t\circ(\phi^{\omega_n}_t)^{-1}$ as well as the boundary condition $(\bar{q}_n^{\omega_n},p_n^{\omega_n})(1)=\phi^{\omega_n}_1((\bar{q}_n^{\omega_n},p_n^{\omega_n})(0))$ are depending on the paths in Wiener space, we introduce for every $\omega_n\in\Omega_n^d$, $n\in\NN$ a corresponding moduli space $\mathcal{M}_n^{\omega_n}=\mathcal{M}_n^{\omega_n}(F,\sigma)$ of Floer curves. In order to be able to employ a maximum principle for proving compactness for moduli spaces of Floer curves, we start with the following standard auxiliary result. 
\begin{lemma}\label{cut-off}
There exists $R>0$ depending on $|W_n(\omega_n,\cdot)|$ such that $K_t^{\omega_n}(q,p)=\frac{1}{2}|p|^2+G_t^{\omega_n}(q,p)$ and $\bar{K}_t^{\omega_n}(q,p)=\frac{1}{2}|p|^2+\bar{G}_t^{\omega_n}(q,p)$, $\bar{G}_t^{\omega_n}(q,p)=\chi_R(|p|)\cdot G_t^{\omega_n}(q,p)$ with the cut-off function $\chi_R:[0,\infty)\to\RR$, $\chi_R(s)=1$ for $s\leq R$, $\chi_R(s)=0$ for $s\geq R+1$, have the same Hamiltonian orbits with symplectic action $\leq \frac{1}{2}(\sigma\cdot W_n(\omega_n,1))^2+4\|F\|_{C^1}$. 
\end{lemma}
\begin{proof}
We start by noting that the dependence on $|W_n(\omega_n,1)|$ is directly related to the given bound on the symplectic action. Since the symplectic action of a Hamiltonian orbit $\bar{u}=(\bar{q},p)$ of $K_t^{\omega_n}$ is given by \[\int_0^1 \left(\frac{1}{2}p(t)^2+p(t)\frac{\partial G_t^{\omega_n}}{\partial p}(\bar{q}(t),p(t))-G_t^{\omega_n}(\bar{q}(t),p(t))\right)\,dt, \] it follows from the fact that $F_t$ and hence $G_t^{\omega_n}$ has bounded $C^1$-norm that the symplectic action grows quadratically with the $L^2$-norm of the $p$-component of the Hamiltonian orbit $(\bar{q}(t),p(t))$. With the bound on the symplectic action in place, it follows that we get a bound on this $L^2$-norm. Using the Hamiltonian equation we get a bound on the Sobolev $W^{1,2}$-norm which in turn leads to a bound of the $p$-component in the supremum norm.      
\end{proof}
\indent For every $1\leq j\leq d$ consider the submanifold $C_j=\TT^{j-1}\times\{0\}\times\TT^{d-j+1}\subset\TT^d$. Using the intersection product of homology classes, note that $[C_1]\cdot\ldots\cdot [C_d]=[\textrm{point}]$, or equivalently, $\textrm{PD}[C_1]\cup\ldots\cup\textrm{PD}[C_d]$ is equal to the canonical volume form on $\TT^d$. Further we introduce, smoothly depending on $\tau>0$, the smooth cut-off function $\varphi_{\tau}:\RR\to [0,1]$ with $\varphi_{\tau}=0$ for $\tau=0$ and $\varphi_{\tau}(s)=1$ for $0<s<\tau$ and $\varphi_{\tau}(s)=0$ for $s<-1$ and $s>\tau+1$ for $\tau >0$ large. Then for every $n\in\NN$ the $\omega_n$-dependent moduli space $\mathcal{M}_n^{\omega_n}$ is defined as the set
\[\mathcal{M}_n^{\omega_n}=\bigcup_{\tau>0}\mathcal{M}_{n,\tau}^{\omega_n},\,\,\mathcal{M}_{n,\tau}^{\omega_n}=\{\widetilde{u}:\RR\times [0,1]\to T^*\TT^d: (*1),(*2),(*3),(*4)\}\] of Floer curves satisfying the $\omega_n$-dependent Floer equation 
\[(*1):\,\, \overline{\partial}^{\omega_n,\tau}_{J,\bar{K}}(\tilde{u})\,:=\,\partial_s\tilde{u}+J_t(\tilde{u})\partial_t\tilde{u}+\nabla K^0(\tilde{u})+\varphi_{\tau}(s)\cdot\nabla \bar{G}^{\omega_n}_t(\tilde{u})=0,\]
the boundary condition \[(*2):\,\,\widetilde{u}(s,1)=\phi^{\omega_n}_1(\widetilde{u}(s,0))\,\,\textrm{for every}\,\,s\in\RR,\] 
and, for $\widetilde{u}=(\widetilde{q},\widetilde{p})$, the asymptotic condition \[(*3):\,\,\widetilde{p}(s,t)\to \sigma\cdot W_n(\omega_n,1), \,\,\widetilde{q}(s,t)\to q+ t\cdot \sigma\cdot W_n(\omega_n,1)\,\textrm{for some}\,q\in\TT^d, \] that is, the Floer curve converges to a solution for the case $F_t=0$ as $s\to\pm\infty$. Here $J_t$ denotes a family of almost complex structure on $T^*\TT^d$ satisfying the periodicity condition $(\phi^{\omega_n}_1)^*J_{t+1}=J_t$. Finally we demand the intersection property \[(*4): \tilde{q}\left(\frac{j}{d}\cdot\tau,0\right)\in C_j\,\,\textrm{for every}\,\,j=1,\ldots,d.\]
\indent In order to prove that $\mathcal{M}^{\omega_n}_n$ carries the structure of a one-dimensional manifold one uses that for every $\tau>0$ the submoduli space $\mathcal{M}^{\omega_n}_{n,\tau}$ is the zero set of the nonlinear Floer operator $\overline{\partial}^{\omega_n,\tau}_{J,\bar{K}}$, viewed as a section in the Banach space bundle $\mathcal{E}^{k,p}_{\omega_n}$ over the Banach manifold $\mathcal{B}^{k+1,p}_{\omega_n}$ with $k=0,1,2,\ldots$ and $p>2$. Here $\mathcal{B}^{k+1,p}_{\omega_n}$ consists of $W^{k+1,p}_{\textrm{loc}}$-maps $\widetilde{u}:\RR\times [0,1]\to T^*\TT^d$ satisfying $(*2)$, $(*3)$, $(*4)$, while the fibre $\mathcal{E}^{k,p}_{\omega_n,\widetilde{u}}$ over $\widetilde{u}\in\mathcal{B}^{k+1,p}_{\omega_n}$ is the linear Banach space $W^{k,p}_{\phi^{\omega_n}_1}(\RR\times [0,1],\RR^{2d})$ of $W^{k,p}$-maps satisfying $(*2)$. In order to prove that $\overline{\partial}^{\omega_n,\tau}_{\bar{K}}$ defines a nonlinear Fredholm operator, one shows that the linearization $D_{\widetilde{u}}: T_{\widetilde{u}}\mathcal{B}^{k+1,p}_{\omega_n}\to \mathcal{E}^{k,p}_{\omega_n,\widetilde{u}}$ is a linear Fredholm operator for every $\widetilde{u}\in\mathcal{M}^{\omega_n}_{n,\tau}$. \\ \indent One of the main ingredients is to show that the gradient $\widetilde{u}\mapsto\nabla {\bar{K}}^{\omega}_t(\widetilde{u})$ defines a bounded linear map from $T_{\widetilde{u}}\mathcal{B}^{k,p}_{\omega_n}$ into $\mathcal{E}^{k,p}_{\omega_n,\widetilde{u}}$, that is, from some $W^{k,p}$-space into another $W^{k,p}$-space.  Since $W_n(\omega_n,\cdot)$ is Lipschitz continuous and hence an element of $W^{1,\infty}([0,1],\RR^d)$, the space of H\"older continuous functions with H\"older exponent $1$, using $K^{\omega_n}_t(\bar{q},p)=H_t(\bar{q}+\sigma\cdot W_n(\omega_n,t),p)$ and the embedding of $W^{1,\infty}$ into $W^{1,p}$ it follows that $\nabla K^{\omega_n}_t$ and hence $\nabla \bar{K}^{\omega_n}_t$ defines a bounded linear map $T_{\widetilde{u}}\mathcal{B}^{k,p}_{\omega_n}$ into $\mathcal{E}^{k,p}_{\omega_n,\widetilde{u}}$ for $k=0,1$. Summarizing we find that $\mathcal{M}^{\omega_n}_n$ is a subset of the Banach manifold $\mathcal{B}^{k+1,p}_{\omega_n}$ with $k=0,1$ and $p>2$, in particular, we get that each $\widetilde{u}$ is an $W^{2,p}$-map and hence at least $C^1$, i.e., differentiable in the classical sense. \\
\indent Now a standard transversality argument shows, possibly after slightly perturbing the family of almost complex structures $J_t$, that $\mathcal{M}_n^{\omega_n}$ is a one-dimensional manifold which is non-empty, since for $\tau=0$ the moduli subspace $\mathcal{M}_n^{\omega_n,0}$ contains precisely one element. Since we employ the cut-off Hamiltonian $\bar{G}_t^{\omega_n}(q,p)=\chi_R(|p|)\cdot G_t^{\omega_n}(q,p)$ instead of the original Hamiltonian $G_t^{\omega_n}$, it follows that the standard $C^0$-bounds for Floer curves in cotangent bundles are available, see e.g. \cite{C}, which in turn implies that also the standard Gromov-Floer compactness is in place. Note that for the latter we use that the energy is uniformly bounded by twice the Hofer norm of $F$, see (\cite{MDSa}, proposition 9.1.4), which in turn is bounded by twice the $C^1$-norm of $F$. Hence it follows that $\mathcal{M}_n^{\omega_n,\tau}$ is non-empty for every $\tau>0$ as $\mathcal{M}_n^{\omega_n}$ is compact up to breaking of cylinders. Compactifying the moduli space $\mathcal{M}_n^{\omega_n}$ in the Gromov-Floer sense, we find in the limit $\tau\to\infty$ for every $j=1,\ldots,d$ a Floer map $\widetilde{u}^{\omega_n,j}_n=(\widetilde{q}^{\omega_n,j}_n,\widetilde{p}^{\omega_n,j}_n):\RR\times [0,1]\to T^*\TT^d$ satisfying the boundary condition $(*2)$ and the other conditions $(*1)$, $(*3)$, and $(*4)$ being replaced by 
\[(*1'):\,\, \overline{\partial}^{\omega_n}_{J,\bar{K}}(\widetilde{u}^{\omega_n,j}_n)\,:=\,\partial_s\widetilde{u}^{\omega_n,j}_n+J_t(\widetilde{u}^{\omega_n,j}_n)\partial_t\widetilde{u}^{\omega_n,j}_n+\nabla \bar{K}^{\omega_n}_t(\widetilde{u}^{\omega_n,j}_n)=0,\]
\[(*3'):\,\,\widetilde{u}^{\omega_n,j}_n(s_{k,\pm},\cdot)\to(\bar{q}^{\omega_n,j}_{n,\pm},p^{\omega_n,j}_{n,\pm})\,\,\textrm{for sequences}\,\,s_{k,\pm}\to\pm\infty\,\,\textrm{as}\,\,k\to\infty ,\]
\[(*4'): \widetilde{q}^{\omega_n,j}_n(0,0)\in C_j.\]
Here $(\bar{q}^{\omega_n,j}_{n,\pm},p^{\omega_n,j}_{n,\pm}):[0,1]\to T^*\TT^d$ is a solution of the Hamiltonian equation for the Hamiltonian $\bar{K}^{\omega_n}_t$ with boundary condition $(\bar{q}^{\omega_n,j}_{n,\pm}(1),p^{\omega_n,j}_{n,\pm}(1))=\phi^{\omega_n}_1(\bar{q}^{\omega_n,j}_{n,\pm}(0),p^{\omega_n,j}_{n,\pm}(0))$. From $(*4')$ it follows that $(\bar{q}^{\omega_n,j}_{n,-},p^{\omega_n,j}_{n,-})$ and $(\bar{q}^{\omega_n,j}_{n,+},p^{\omega_n,j}_{n,+})$ can be distinguished by their symplectic action given by \[\mathcal{L}_n^{\omega_n}(\bar{q}^{\omega_n,j}_{n,\pm},p^{\omega_n,j}_{n,\pm})=\int_0^1 \left( d\bar{K}^{\omega_n}_t\cdot p\frac{\partial}{\partial p}-\bar{K}^{\omega_n}_t\right)(\bar{q}^{\omega_n,j}_{n,\pm}(t),p^{\omega_n,j}_{n,\pm}(t))\,dt\] 
\indent Due to the asymptotic condition $(*3)$ in the definition of the moduli spaces $\mathcal{M}_n^{\omega_n}$, it follows from (\cite{MDSa}, proposition 9.1.4) that the above actions differ from the symplectic action $\frac{1}{2}(\sigma\cdot W_n(\omega_n,1))^2$ of the asymptotic orbit for $F_t\equiv 0$ by at most twice the Hofer norm $|||F|||$. Since the latter is bounded by twice the $C^1$-norm of $F$, it follows from our choice of auxiliary Hamiltonian in \Cref{cut-off} that $(\bar{q}^{\omega_n,j}_{n,-},p^{\omega_n,j}_{-,n})$ and $(\bar{q}^{\omega_n,j}_{+,n},p^{\omega_n,j}_{+,n})$ are indeed Hamiltonian orbits for the original Hamiltonian $K_t^{\omega_n}$ with $K_t^{\omega_n}(q,p)=\frac{1}{2}|p|^2+G_t^{\omega_n}(q,p)$ and we have \[\mathcal{L}_n^{\omega_n}(\bar{q}^{\omega_n,j}_{n,\pm},p^{\omega_n,j}_{n,\pm})=\int_0^1 \left( dK^{\omega_n}_t\cdot p\frac{\partial}{\partial p}-K^{\omega_n}_t\right)(\bar{q}^{\omega_n,j}_{n,\pm}(t),p^{\omega_n,j}_{n,\pm}(t))\,dt\]
More precisely we have \[\mathcal{L}_n^{\omega_n}(\bar{q}^{\omega_n,1}_{n,-},p^{\omega_n,1}_{n,-})<\mathcal{L}_n^{\omega_n}(\bar{q}^{\omega_n,1}_{n,+},p^{\omega_n,1}_{n,+})\leq \mathcal{L}_n^{\omega_n}(\bar{q}^{\omega_n,2}_{n,-},p^{\omega_n,2}_{n,-})<\ldots<\mathcal{L}_n^{\omega_n}(\bar{q}^{\omega_n,d}_{n,+},p^{\omega_n,d}_{n,+})\] which in turn implies that there at least $d+1$ different contractible solutions.\\
\\  
\noindent\emph{Tight family of measures and Gromov-Floer compactness:} It remains to show that the random Hamiltonian orbits $\bar{u}^j_{\pm}=(\bar{u}^j_{n,\pm})_{n\in\NN}$, with $\bar{u}^j_{n,\pm}(\omega_n,t)=\bar{u}^{\omega_n,j}_{n,\pm}(t)=(\bar{q}^{\omega_n,j}_{n,\pm}(t),p^{\omega_n,j}_{n,\pm}(t))$ for $(\omega_n,t)\in\Omega_n^d\times [0,1]$, converge in distribution as $n\to\infty$ in the sense that the corresponding Dirac measures $\bar{\mu}^{\bar{u}^j_\pm}_n$ converge, possibly after passing to a subsequence. Furthermore, in order to show that the resulting limiting Borel measures $\bar{\mu}^{\bar{u}^j_\pm}$ are different, we further show a corresponding statement for the families of Floer curves $\widetilde{u}^{\omega_n,j}_n$ connecting $\bar{u}^{\omega_n,j}_{n,-}$ and $\bar{u}^{\omega_n,j}_{n,+}$ for each $j=1,\ldots,d$.
\begin{lemma}\label{tight}
    For every $\epsilon>0$ there exists a compact subset $C^1_{\epsilon}$ of $C^1([0,1],T^*\TT^d)$ such that for each $j=1,\ldots,d$ we have $\bar{\mu}^{\bar{u}^j_\pm}_n(C^1_{\epsilon})\geq 1-\epsilon$ for $n$ sufficiently large. In particular, after passing to a subsequence, $\bar{\mu}^{\bar{u}^j_\pm}_n$ converges to some Borel measure $\bar{\mu}^{\bar{u}^j_\pm}$ on $C^1([0,1],T^*\TT^d)$ as $n\to\infty$. 
\end{lemma}
Note that the first half of the statement can be rephrased as (asymptotical) tightness of the family of probability measures. Since tight families of probability measures are well-known to be compact, see e.g. (\cite{Bi}, theorem 25.10), the second half of the statement indeed follows from the first.
\begin{proof}
Let $\epsilon>0$ be arbitrary. By \Cref{regularity} we know that the space $W^{\frac{1}{4},\infty}([0,1],\TT^d)$ of Hölder continuous functions with Hölder exponent $\frac{1}{4}<\frac{1}{2}$ has full Wiener measure $\mu$. Denoting by $W^{\frac{1}{4},\infty}_B([0,1],\TT^d)$ the subspace of functions with $W^{\frac{1}{4},\infty}$-norm less than or equal to $B$, and using that $\mu$ is the limit of the Dirac measures $\mu_n=\nu_n\circ W_n^{-1}$, we find $B>0$ and $n_0\in\NN$ such that $\mu_n(W^{\frac{1}{4},\infty}_B([0,1],\TT^d))\geq 1-\epsilon$ for $n\geq n_0$. Now all that remains to be shown is that there exists $\bar{B}>0$ with the following property for all $n\in\NN$, $\omega_n\in\Omega_n^d$: If $W_n(\omega_n,\cdot)\in W^{\frac{1}{4},\infty}_B([0,1],\TT^d)$, then $\bar{u}^{\omega_n,j}_{n,\pm}\in C^1_{\epsilon}$ for $j=1,\ldots,d$ with the compact subset $C^1_{\epsilon}:=W^{1\frac{1}{4},\infty}_{\bar{B}}([0,1],T^*\TT^d)$ of all maps in $C^1([0,1],T^*\TT^d)$ with $W^{1\frac{1}{4},\infty}$-norm less than or equal to $\bar{B}$. As a first step we observe that, since the $W^{\frac{1}{4},\infty}$-norm of $W_n(\omega_n,\cdot)$ dominates its $C^0$-norm, it follows from \Cref{cut-off} that there is a bound for the $C^0$-norm of $\bar{u}^{\omega_n,j}_{n,\pm}$ which just depends on the chosen $B>0$. On the other hand, using $J_t(\bar{u}^{\omega_n,j}_{n,\pm})\partial_t\bar{u}^{\omega_n,j}_{n,\pm}+\nabla H_t(\bar{u}^{\omega_n,j}_{n,\pm}+(\sigma W_n(\omega_n,\cdot),0))=0$ it follows that the $W^{\frac{1}{4},\infty}$-norm of $\partial_t\bar{u}^{\omega_n,j}_{n,\pm}$ is uniformly bounded as well, again depending on $B>0$. 
\end{proof}
Recall that we have shown above for every $n\in\NN$, $\omega_n\in\Omega_n^d$ that the orbits $\bar{u}^{\omega_n,j}_{n,\pm}$ are pairwise different as they can be ordered via their symplectic action. Here the crucial strict inequality is that for each $j=1,\ldots,d$ we have $\mathcal{L}_n^{\omega_n}(\bar{q}^{\omega_n,j}_{n,-},p^{\omega_n,j}_{n,-})<\mathcal{L}_n^{\omega_n}(\bar{q}^{\omega_n,j}_{n,+},p^{\omega_n,j}_{n,+})$, which follows from the existence of the Floer map $\widetilde{u}^{\omega_n,j}_n:\RR^2\to T^*\TT^d$ connecting $\bar{u}^{\omega_n,j}_{n,-}=(\bar{q}^{\omega_n,j}_{n,-},p^{\omega_n,j}_{n,-})$ and $\bar{u}^{\omega_n,j}_{n,+}=(\bar{q}^{\omega_n,j}_{n,+},p^{\omega_n,j}_{n,+})$ in the sense of $(*3')$. In order to establish that the symplectic actions for $\bar{u}^{\omega_n,j}_{n,-}$ and $\bar{u}^{\omega_n,j}_{n,+}$ are different from each other, we crucially use that the Floer curve must be nontrivial due to $(*4')$, i.e., it must intersect a given homology cycle. In order to see that this argument carries through to the limit as $n\to\infty$, we show that the Floer maps $\widetilde{u}^{\omega_n,j}_n$ themselves converge in distribution.\\
\indent But before we can state the corresponding statement and prove it, we first need the following technical result about the Cauchy-Riemann operator $\overline{\partial}_J(\widetilde{u})=\partial_s\widetilde{u}+J_t(\widetilde{u})\partial_t\widetilde{u}$.
\begin{lemma}\label{regularity} Fix some real number $0\leq\alpha\leq 1$ and $p>2$. For every $S>0$ there exists $c>0$ such that for every $W^{1,p}$-map $\widetilde{u}:[-S,+S]\times[0,1]\to T^*\TT^d$ we have \[\|\widetilde{u}\|_{\alpha+1,p}\leq c\left(\|\overline{\partial}_J(\widetilde{u})\|_{\alpha,p}+\|\widetilde{u}\|_{0,p}\right),\] where $\|\cdot\|_{\alpha,p}$ denotes the $W^{\alpha,p}$-norm.\end{lemma}
\begin{proof}
In (\cite{MDSa}, B.3.4) it is shown that the result holds for $\alpha=0$ and $\alpha=1$. Now it follows from the classical interpolation theory, see (\cite{AF}, 7.22) or (\cite{BL}, definition 2.4.1), that the same holds true for every $0\leq\alpha\leq 1$, i.e.,  when $\overline{\partial}_J$ is viewed as a map from $W^{\alpha+1,p}$ to $W^{\alpha,p}$. Here we would like to remark that fractional Sobolev spaces $W^{\alpha,p}$ and $W^{\alpha+1,p}$ with a non-integer number of weak derivatives can be either defined using Fourier transform on the space of tempered distributions or as interpolation space in the sense of (\cite{AF}, 7.5.7) or (\cite{BL}, definition 2.4.1)  between $L^p$ and $W^{1,p}$ or $W^{1,p}$ and $W^{2,p}$, respectively, see (\cite{BL}, theorem 6.4.5).
\end{proof}
Now let $S>0$ be chosen arbitrary and fixed. Observe that for each $n\in\NN$ the restricted Floer maps $\widetilde{u}^{\omega_n,j}_n: [-S,+S]\times [0,1]\to T^*\TT^d$, $\omega_n\in\Omega_n^d$  now define Dirac measures $\widetilde{\mu}^{\widetilde{u}^j}_n$ on $C^1([-S,+S]\times [0,1],T^*\TT^d)$. We can prove the following Gromov-Floer analogue of \Cref{tight}.
\begin{lemma}\label{tight-Floer}
    For every $\epsilon>0$ there exists a compact subset $\widetilde{C}^1_{\epsilon}$ of $C^1([-S,+S]\times[0,1],T^*\TT^d)$ such that for each $j=1,\ldots,d$ we have $\widetilde{\mu}^{\widetilde{u}^j}_n(\widetilde{C}^1_{\epsilon})\geq 1-\epsilon$ for $n$ sufficiently large. In particular, after passing to a subsequence, $\widetilde{\mu}^{\widetilde{u}^j}_n$ converges to some Borel measure $\widetilde{\mu}^{\widetilde{u}^j}$ on $C^1([-S,+S]\times[0,1],T^*\TT^d)$ as $n\to\infty$. 
\end{lemma}
\begin{proof}
As in the proof of \Cref{tight} it suffices to show is that there exists $\bar{B}>0$ with the following property for all $n\in\NN$, $\omega_n\in\Omega_n^d$: If $W_n(\omega_n,\cdot)\in W^{\frac{1}{4},\infty}_B([0,1],\TT^d)$ with $B>0$ from the proof of \Cref{tight}, then $\widetilde{u}^{\omega_n,j}_n\in \widetilde{C}^1_{\epsilon}$, $j=1,\ldots,d$ with the compact subset $\widetilde{C}^1_{\epsilon}:=W^{1\frac{1}{4},p}_{\bar{B}}([-S,+S]\times [0,1],T^*\TT^d)$ of all maps in $C^1([-S,+S]\times[0,1],T^*\TT^d)$ with $W^{1\frac{1}{4},p}$-norm less than or equal to $\bar{B}$. Note that here $p>2$ is chosen large enough such that $W^{1\frac{1}{4},p}$ embeds compactly into $C^1$. 
As a first step we observe that, since we employ the cut-off Hamiltonian $\bar{K}_t^{\omega_n}(q,p)=\frac{1}{2}|p|^2+\bar{G}_t^{\omega_n}(q,p)$, $\bar{G}_t^{\omega_n}(q,p)=\chi_R(|p|)\cdot G_t^{\omega_n}(q,p)$ instead of the original Hamiltonian $K_t^{\omega_n}(q,p)=\frac{1}{2}|p|^2+G_t^{\omega_n}(q,p)$, it follows that the standard $C^0$-bounds for Floer curves in cotangent bundles from \cite{C} are available. In particular, there is a bound for the $C^0$-norm of $\widetilde{u}^{\omega_n,j}_n$ which just depends on the chosen $B>0$ in view of the choice of $R>0$ in \Cref{cut-off}. While the uniform energy bound given by twice the Hofer norm of $F$ is sufficient to establish uniform $L^2$-bounds for the first derivatives $T\widetilde{u}^{\omega_n,j}_n$, due to the fact that $W^{1,2}([-S,+S]\times [0,1],T^*\TT^d)$ does not embed into $C^0([-S,+S]\times [0,1],T^*\TT^d)$, the latter is not sufficient. However, together with the $C^0$-bounds mentioned above, we now again make use of the fact that the standard Gromov-Floer compactness is in place, which in turn shows that a uniform $C^0$-bound for the first derivatives $T\widetilde{u}^{\omega_n,j}_n$ can be established. For the proof, note that, if the latter bound would not exist, then within the set of all restricted Floer maps $\widetilde{u}^{\omega_n,j}_n$ one would find a sequence which would develop a holomorphic sphere in some point in $[-S,+S]\times [0,1]$, which in turn is excluded by the exactness of the symplectic form on $T^*\TT^d$. Since $\widetilde{u}^{\omega_n,j}_n$ is hence uniformly bounded with respect to the $C^1$-norm and $W_n(\omega_n,\cdot)$ is uniformly bounded with respect to the Hölder $W^{\frac{1}{4},\infty}$-norm, their sum $\widetilde{u}^{\omega_n,j}_n+(\sigma W_n(\omega_n,\cdot),0)$ is uniformly bounded with respect to the $W^{\frac{1}{4},p}$-norm for every $p>2$. 
Since each Floer map $\widetilde{u}^{\omega_n,j}_n$ satisfies the Floer equation $\overline{\partial}_J(\widetilde{u}^{\omega_n,j}_n)+\nabla H_t(\widetilde{u}^{\omega_n,j}_n+(\sigma W_n(\omega_n,\cdot),0))=0$, it follows from \Cref{regularity} that $\widetilde{u}^{\omega_n,j}_n$ is indeed uniformly bounded with respect to the $W^{1\frac{1}{4},p}$-norm. \end{proof}

In order to finish the proof of \Cref{main}, we use the standard (in)equality relating the energy of the restricted Floer maps $\widetilde{u}^{\omega_n,j}_n$ with the symplectic action of their asymptotic orbits $\bar{u}^{\omega_n,j}_{n,\pm}$, \[\int_{-S}^{+S}\int_0^1 \left|\partial_s\widetilde{u}^{\omega_n,j}_n(s,t)\right|^2\,dt\,ds \,\leq\, \mathcal{L}_n^{\omega_n}(\bar{q}^{\omega_n,j}_{n,+},p^{\omega_n,j}_{n,+}) - \mathcal{L}_n^{\omega_n}(\bar{q}^{\omega_n,j}_{n,-},p^{\omega_n,j}_{n,-}),\] together with 
\begin{eqnarray*} \EE\mathcal{L}_n^{\omega_n}(\bar{q}^{\omega_n,j}_{n,\pm},p^{\omega_n,j}_{n,\pm})&=&\EE\int_0^1 \left( dK^{\omega_n}_t\cdot p\frac{\partial}{\partial p}-K^{\omega_n}_t\right)(\bar{q}^{\omega_n,j}_{n,\pm}(t),p^{\omega_n,j}_{n,\pm}(t))\,dt\\&=&\EE\int_0^1 \left( dH_t\cdot p\frac{\partial}{\partial p}-H_t\right)(q^{\omega_n,j}_{n,\pm}(t),p^{\omega_n,j}_{n,\pm}(t))\,dt\\&=&\int_{C^0}\int_0^1 \left(p(t)\frac{\partial H_t}{\partial p}(q(t),p(t))- H_t(q(t),p(t))\right)\,dt\,d\mu_n^{u^{j,\pm}}\\&=&\int_0^1\int_{T^*\TT^d} \left(p\frac{\partial H_t}{\partial p}(q,p)- H_t(q,p)\right)\,\rho^{u^{j,\pm}}_n(t,q,p) \,dp\,dq \,dt
\end{eqnarray*}
to deduce that the energy of $\widetilde{\mu}^{\widetilde{u}^j}_n$, that is, the expectation value of the energy of the Floer curves $\widetilde{u}^{\omega_n,j}_n$, \[\EE\int_{-S}^{+S}\int_0^1 \left|\partial_s\widetilde{u}^{\omega_n,j}_n(s,t)\right|^2\,dt\,ds=\int_{C^1}\int_{-S}^{+S}\int_0^1 \left|\partial_s u(s,t)\right|^2\,dt\,ds\,d\widetilde{\mu}^j_n\] is bounded from above by the difference $\mathcal{L}(\rho^{u^{j,+}}_n)-\mathcal{L}(\rho^{u^{j,-}}_n)$ of the symplectic actions of $\rho^{u^{j,\pm}}_n$, \[\mathcal{L}(\rho^{u^{j,\pm}}_n)=\int_0^1\int_{T^*\TT^d} \left(p\frac{\partial H_t}{\partial p}(q,p)- H_t(q,p)\right)\,\rho^{u^{j,\pm}}_n(t,q,p) \,dp\,dq \,dt.\]
Taking limits as $n\to\infty$, it follows that the energy of $\widetilde{\mu}^{\widetilde{u}^j}$, \[\int_{C^1}\int_{-S}^{+S}\int_0^1 \left|\partial_s u(s,t)\right|^2\,dt\,ds\,d\widetilde{\mu}^j\] is bounded from above by the difference $\mathcal{L}(\rho^{u^{j,+}})-\mathcal{L}(\rho^{u^{j,-}})$ of the symplectic actions of $\rho^{u^{j,\pm}}$. With this we can deduce the strict inequality $\mathcal{L}(\rho^{u^{j,-}})<\mathcal{L}(\rho^{u^{j,+}})$: 
Assuming to the contrary that $\mathcal{L}(\rho^{u^{j,-}})\geq\mathcal{L}(\rho^{u^{j,+}})$, that is, $\mathcal{L}(\rho^{u^{j,-}})=\mathcal{L}(\rho^{u^{j,+}})$, it would follow that \[\int_{C^1}\int_{-S}^{+S}\int_0^1 \left|\partial_s u(s,t)\right|^2\,dt\,ds\,d\widetilde{\mu}^j=0,\] that is, the limiting Borel measure $\widetilde{\mu}^j$ would have full measure on maps in $C^1([-S,+S]\times[0,1],T^*\TT^d)$ which are independent of $s\in [-S,+S]$. But the latter is impossible in view of the intersection condition $(*4')$ imposed on all Floer curves.\\ 
\indent Summarizing, we find that \[\mathcal{L}(\rho^{u^{1,-}})<\mathcal{L}(\rho^{u^{1,+}})\leq \mathcal{L}(\rho^{u^{2,-}})<\ldots<\mathcal{L}(\rho^{u^{d-1,+}})\leq \mathcal{L}(\rho^{u^{d,-}})<\mathcal{L}(\rho^{u^{d,+}})\] which in turn implies that there at least $d+1$ different solutions of the Hamiltonian Fokker-Planck equation.\\

\end{document}